\documentclass[12pt]{amsart}%
\usepackage{amscd,amsmath,latexsym,amsthm,amsfonts,amssymb,graphicx,color,geometry}
\usepackage{enumerate}
\usepackage[pdftex,plainpages=false,pdfpagelabels,
colorlinks=true,citecolor=blue,linkcolor=blue,urlcolor=black,
filecolor=black,bookmarksopen=true,unicode=true]{hyperref}
\usepackage[norefs,nocites]{refcheck}
\usepackage{amsmath}
\usepackage{amsfonts}
\usepackage{amssymb}
\usepackage[T1]{fontenc}
\usepackage{graphicx}
\usepackage{xcolor}
\usepackage{colortbl}%
\providecommand{\U}[1]{\protect\rule{.1in}{.1in}}
\newtheorem{theorem}{Theorem}[section]

\newtheorem{corollary}[theorem]{Corollary}

\numberwithin{equation}{section}
\begin{document}
\title[Spaceability of sets of operators between sequence spaces]{Spaceability of the sets of surjective and injective operators between
sequence spaces}

\begin{abstract}
We investigate algebraic structures within sets of surjective and injective
linear operators between sequence spaces, completing results of Aron et al.

\end{abstract}
\author[D. Diniz]{Diogo Diniz}
\address{Unidade Acad\^{e}mica de Matem\'{a}tica e Estat\'{\i}stica \\
Universidade Federal de Campina Grande \\
58109-970 - Campina Grande, Brazil.}
\email{diogodpss@gmail.com}
\author[V. F\'{a}varo]{Vin\'icius V. F\'{a}varo}
\address{Faculdade de Matem\'{a}tica \\
Universidade Federal de Uberl\^{a}ndia \\
38400-902 - U\-ber\-l\^{a}n\-dia, Brazil.}
\email{vvfavaro@gmail.com}
\author[D. Pellegrino]{Daniel Pellegrino}
\address{Departamento de Matem\'{a}tica \\
Universidade Federal da Para\'{\i}ba \\
58.051-900 - Jo\~{a}o Pessoa, Brazil.}
\email{dmpellegrino@gmail.com}
\author[A. Raposo Jr.]{Anselmo Raposo Jr.}
\address{Departamento de Matem\'{a}tica \\
Universidade Federal do Maranh\~{a}o \\
65085-580 - S\~{a}o Lu\'{\i}s, Brazil.}
\email{anselmo.junior@ufma.br}
\thanks{D. Diniz was partially supported by CNPq grants No.~303822/2016-3,
No.~406401/2016-0 and No.~421129/2018-2 and Grant
2019/0014 Paraiba State Research Foundation (FAPESQ)}
\thanks{V. F\'{a}varo was supported by CNPq 310500/2017-6 and FAPEMIG Grant PPM-00217-18}
\thanks{D. Pellegrino was partially supported by CNPq 307327/2017-5 and Grant
2019/0014 Paraiba State Research Foundation (FAPESQ)}
\maketitle

\section{Introduction}

If $V$ is a vector space and $\alpha$ is a cardinal number, a subset $A$ of
$V$ is called $\alpha\text{-lineable}$ in $V$ if $A\cup\left\{  0\right\}  $
contains an $\alpha$-dimensional linear subspace $W$ of $V$. When $V$ has a
topology and the subspace $W$ can be chosen to be closed, we say that $A$ is
spaceable. This line of research has its starting point with the seminal paper
\cite{Aron} by Aron, Gurariy, and Seoane-Sep\'{u}lveda and nowadays has been
successfully explored in several research branches, being studied in various
contexts with increasingly relevant applications in areas such as
norm-attaining operators, multilinear forms, homogeneous polynomials, sequence
spaces, holomorphic mappings, absolutely summing operators, Peano curves,
fractals, topological dynamical systems and many others (see, for instance,
{\cite{Nacib, book, bba, Bernal, Pellegrino2, BF, CFS, cariellojfa,
SSJB,FGMR}} and the references therein).

From now on all vector spaces are considered over a fixed scalar field
$\mathbb{K}$ which can be either $\mathbb{R}$ or $\mathbb{C}$. For any set $X$
we shall denote by $\operatorname*{card}\left(  X\right)  $ the cardinality of
$X$; in particular, we denote $\mathfrak{c}=\operatorname*{card}\left(
\mathbb{R}\right)  $ and $\aleph_{0}=\operatorname*{card}\left(
\mathbb{N}\right)  $.

In this paper we are interested in lineability and spaceability properties of
sets of injective and surjective continuous linear operators between sequence
spaces. The following results were recently proved in {\cite{bb}:}

\begin{theorem}
{\cite[Theorem 4.1 and Corollary 3.4]{bb}} The set
\begin{equation}
\mathcal{S}=\left\{  T\colon\ell_{p}\rightarrow\ell_{p}:T\text{ is linear,
continuous and surjective}\right\}  \label{098}%
\end{equation}
is spaceable in $\mathcal{L}\left(  \ell_{p};\ell_{p}\right)  $ for all
$p\in\lbrack1,\infty]$ and the set
\begin{equation}
\mathcal{I}=\left\{  T:c_{0}\rightarrow c_{0}:T\text{ is linear, continuous
and injective}\right\}  \label{099}%
\end{equation}
is spaceable in $\mathcal{L}\left(  c_{0},c_{0}\right)  $.
\end{theorem}

For surjective operators, the proof has some matrix arguments split for
different choices of $p$ and duality. In the case of injective operators, the
argument used in the proof is strongly connected with the $\sup$ norm of
$c_{0}$ having no immediate adaptation to $\ell_{p}$ spaces (see
{\cite[Theorem 3.3]{bb}).} The main results of the present paper extend, with
different techniques, the above results to a wide class of sequence spaces.
For instance, (\ref{098}) is extended to a class of sequence spaces
encompassing the spaces $\ell_{p}^{u}(X)$ of unconditionally $p$-summable
sequences; and (\ref{099}) is extended to a very general class of sequence
spaces containing $\ell_{p}^{u}(X)$, the spaces of weakly $p$-summable
sequences $\ell_{p}^{w}(X)$, among others. These classes of sequence spaces
will be formally defined in the beginning of Section 2. Our main results read
as follows:

\begin{theorem}
\label{t1}Let $E$ be a $c_{00}$-dense standard Banach sequence space. The set
\[
\mathcal{S}=\left\{  T\colon E\rightarrow E:T\text{ is linear, continuous and
surjective}\right\}
\]
is spaceable in $\mathcal{L}\left(  E;E\right)  $.
\end{theorem}

\begin{theorem}
\label{t2}Let $V$ be an infinite dimensional Banach space and let $E$ be a
standard Banach sequence space. The set%
\[
\mathcal{I}=\left\{  T\colon V\rightarrow E:T\text{\ is linear, continuous and
injective}\right\}
\]
is either empty or spaceable in $\mathcal{L}\left(  V;E\right)  $.
\end{theorem}

As a matter of fact, in Theorem \ref{t2} we prove an even stronger result: we
show that $\mathcal{I}$ is $\left(  1,\mathfrak{c}\right)  $-spaceable,
according to the notion recently introduced in \cite{cc} which shall be
recalled later.

The paper is organized as follows: in Section 2 we introduce the definition of
standard Banach sequence spaces and prove Theorem \ref{t1} and some
corollaries. In Section 3 we prove Theorem \ref{t2} and present some consequences.

\section{Spaceability of continuous surjective linear operators}

Let $X\neq\left\{  0\right\}  $ be a Banach space. By a standard Banach
sequence space over $X$ we mean an infinite-dimensional Banach space $E$ of
$X$-valued sequences enjoying the following conditions:

\begin{enumerate}
\item[(i)] There is $C>0$ such that
\[
\left\Vert x_{j}\right\Vert _{X}\leq C\left\Vert x\right\Vert _{E}%
\]
for every $x=\left(  x_{j}\right)  _{j=1}^{\infty}\in E$ and all
$j\in\mathbb{N}$.

\item[(ii)] If $x=\left(  x_{j}\right)  _{j=1}^{\infty}\in E$ and $\left(
x_{n_{k}}\right)  _{k=1}^{\infty}$ is a subsequence of $x$ then $\left(
x_{n_{k}}\right)  _{k=1}^{\infty}\in E$ and
\[
\left\Vert \left(  x_{n_{k}}\right)  _{k=1}^{\infty}\right\Vert _{E}%
\leq\left\Vert x\right\Vert _{E}\text{.}%
\]

\item[(iii)] If $\left(  x_{j}\right)  _{j=1}^{\infty}\in E$ and $\left\{
n_{1}<n_{2}<n_{3}<\cdots\right\}  $ is an infinite subset of $\mathbb{N}$,
then the $X$-valued sequence $\left(  y_{j}\right)  _{j=1}^{\infty}$ defined
as%
\[
y_{j}=\left\{
\begin{array}
[c]{ll}%
x_{i}\text{,} & \text{if }j=n_{i}\text{,}\\
0\text{,} & \text{otherwise,}%
\end{array}
\right.
\]
belongs to $E$ and
\[
\left\Vert \left(  y_{j}\right)  _{j=1}^{\infty}\right\Vert _{E}\leq\left\Vert
\left(  x_{j}\right)  _{j=1}^{\infty}\right\Vert _{E}\text{.}%
\]

\end{enumerate}

\noindent From now on we shall call standard Banach sequence space for a
standard Banach sequence space over some Banach space $X$ and, when
$c_{00}\left(  X\right)  $ is dense in $E$, we say that $E$ is a $c_{00}%
$-dense standard Banach sequence space.

Notice that (i) ensures that the $m$-th projection over $X$%
\begin{align*}
\pi_{m}\colon E  &  \rightarrow X\\
\left(  x_{j}\right)  _{j=1}^{\infty}  &  \mapsto x_{m}%
\end{align*}
is a continuous linear operator for every $m$. Therefore, pointwise
convergence implies coordinatewise convergence. Also, (ii) yields that if
$x\in E$ then each subsequence of $x$ belongs to $E$. Moreover, if
$\mathbb{N}^{\prime}$ is an infinite subset of positive integers, then the
linear operator%
\begin{align*}
T\colon E  &  \rightarrow E\\
\left(  x_{j}\right)  _{j=1}^{\infty}  &  \mapsto\left(  x_{k}\right)
_{k\in\mathbb{N}^{\prime}}%
\end{align*}
is well-defined and continuous.

Finally, from (iii) we have that$\ $if $\mathbb{N}^{\prime}=\left\{
n_{1}<n_{2}<n_{3}<\cdots\right\}  $ is an infinite subset of positive integers
then the linear operator%
\begin{align*}
S\colon E  &  \rightarrow E\\
\left(  x_{j}\right)  _{j=1}^{\infty}  &  \mapsto\left(  y_{j}\right)
_{j=1}^{\infty}%
\end{align*}
where
\[
y_{j}=\left\{
\begin{array}
[c]{ll}%
x_{i}\text{,} & \text{if }j=n_{i}\in\mathbb{N}^{\prime}\text{,}\\
0\text{,} & \text{otherwise,}%
\end{array}
\right.
\]
is continuous and well-defined. In particular, if
\begin{align}
F^{n}\colon E  &  \rightarrow E\label{hhhhh}\\
\left(  x_{j}\right)  _{j=1}^{\infty}  &  \mapsto(\underset{n}{\underbrace
{0,\ldots,0}},x_{1},x_{2},x_{3},\ldots)\nonumber
\end{align}
is the forward $n$-shift then $F^{n}$ is continuous and well-defined.

Now we are ready to prove Theorem \ref{t1}. Splitting the natural numbers in
disjoint infinite subsets $\mathbb{N}_{1},\mathbb{N}_{2},\ldots$ and denoting
the elements of $\mathbb{N}_{k}$ as
\[
\mathbb{N}_{k}=\{n_{k,1},n_{k,2},n_{k,3},\ldots\}
\]
we define, for all $k$, the operators
\begin{align*}
S_{k}\colon E &  \rightarrow E\\
\left(  a_{j}\right)  _{j=1}^{\infty} &  \mapsto\left(  a_{j}\right)
_{j\in\mathbb{N}_{k}}\text{.}%
\end{align*}
By (ii), for all $k$, we have
\[
\Vert S_{k}\Vert=\sup_{\left\Vert \left(  a_{j}\right)  _{j=1}^{\infty
}\right\Vert _{E}\leq1}\Vert S_{k}(a_{j})_{j=1}^{\infty}\Vert_{E}\leq1\text{.}%
\]
It is obvious that $S_{k}$ is surjective for all $k$. In fact, given
$c=(c_{j})_{j=1}^{\infty}\in E$, note that by (iii) we have that
$a=(a_{j})_{j=1}^{\infty}$ defined as
\[
a_{j}=\left\{
\begin{array}
[c]{ll}%
c_{i}\text{,} & \text{if }j=n_{k,i}\in\mathbb{N}_{k}\text{,}\\
0\text{,} & \text{otherwise,}%
\end{array}
\right.
\]
belongs to $E$ and it is obvious that $S(a)=c$. It is also simple to verify
that non trivial linear combinations of $S_{k}$ are also surjective. Let%
\[
S=%
{\textstyle\sum\limits_{i=1}^{n}}
b_{i}S_{i}%
\]
be a non trivial linear combination of $S_{1},\ldots,S_{n}$ and let $k$ be an
index such that $b_{k}\neq0$. Note that, given $c=\left(  c_{j}\right)
_{j=1}^{\infty}\in E$, if we consider the sequence $a=\left(  a_{j}\right)
_{j=1}^{\infty}$ defined as%
\[
a_{j}=\left\{
\begin{array}
[c]{ll}%
b_{k}^{-1}c_{i}\text{,} & \text{if }j=n_{k,i}\in\mathbb{N}_{k}\text{,}\\
0\text{,} & \text{if }j\notin\mathbb{N}_{k}\text{,}%
\end{array}
\right.
\]
then $S_{i}\left(  a\right)  =0$ if $i\neq k$ and
\[
S\left(  a\right)  =b_{k}S_{k}\left(  a\right)  =c\text{.}%
\]
As a consequence, $\left\{  S_{k}:k\in\mathbb{N}\right\}  $ is a linearly
independent subset of $\mathcal{L}\left(  E,E\right)  $. In fact, any non
trivial linear combination of elements of $\left\{  S_{k}:k\in\mathbb{N}%
\right\}  $ is surjective and, in particular, different from $0$. Now
consider
\begin{align*}
\Psi\colon\ell_{1} &  \rightarrow\mathcal{L}\left(  E;E\right)  \\
\left(  b_{k}\right)  _{k=1}^{\infty} &  \mapsto%
{\textstyle\sum\limits_{k=1}^{\infty}}
b_{k}S_{k}\text{.}%
\end{align*}
Note that $\Psi$ is well-defined. In fact, since $\Vert S_{k}\Vert\leq1$, we
have
\[%
{\textstyle\sum\limits_{k=1}^{\infty}}
\left\Vert b_{k}S_{k}\right\Vert \leq%
{\textstyle\sum\limits_{k=1}^{\infty}}
\left\vert b_{k}\right\vert <\infty
\]
and since $\mathcal{L}\left(  E;E\right)  $ is complete, it follows that $%
{\textstyle\sum\limits_{k=1}^{\infty}}
b_{k}S_{k}$ belongs to $\mathcal{L}\left(  E;E\right)  $.

Also, the same argument used before for finite sums is straightforwardly
adapted to prove that $%
{\textstyle\sum\limits_{k=1}^{\infty}}
b_{k}S_{k}$ is always surjective whenever $\left(  b_{k}\right)
_{k=1}^{\infty}\neq0$ and, in particular, $\Psi$ is injective.

It remains to prove the spaceability of the set $\mathcal{S}$ defined in
Theorem \ref{t1}. Let us denote by $\operatorname{Im}(\Psi)$ the image of
$\Psi$ and by $\overline{\operatorname{Im}(\Psi)}$ its closure in
$\mathcal{L}\left(  E;E\right)  $. Let $0\neq S\in\overline{\operatorname{Im}%
(\Psi)}$; we only need to prove that $S$ is surjective. Consider a sequence of
elements $g_{n}=\sum\limits_{k=1}^{\infty}b_{k}^{(n)}S_{k}\in\operatorname{Im}%
\left(  \Psi\right)  $ converging to $S$. Hence, for each $a=\left(
a_{j}\right)  _{j=1}^{\infty}\in E$ we have%
\[
S\left(  a\right)  =\lim\limits_{n\rightarrow\infty}%
{\textstyle\sum\limits_{k=1}^{\infty}}
b_{k}^{(n)}S_{k}\left(  a\right)
\]
and, since $\pi_{m}$ is continuous for all $m\in\mathbb{N}$,
\[
\lim_{n\rightarrow\infty}\pi_{m}\left(
{\textstyle\sum\limits_{k=1}^{\infty}}
b_{k}^{(n)}S_{k}\left(  a\right)  \right)  =\pi_{m}\left(  S\left(  a\right)
\right)  \text{.}%
\]
So, we have
\begin{equation}
S\left(  a\right)  =\left(  \lim_{n\rightarrow\infty}\pi_{1}\left(
{\textstyle\sum\limits_{k=1}^{\infty}}
b_{k}^{(n)}S_{k}\left(  a\right)  \right)  ,\lim_{n\rightarrow\infty}\pi
_{2}\left(
{\textstyle\sum\limits_{k=1}^{\infty}}
b_{k}^{(n)}S_{k}\left(  a\right)  \right)  ,\ldots\right)  \text{,}
\label{777}%
\end{equation}
for all $a=\left(  a_{j}\right)  _{j=1}^{\infty}\in E$. Since
\begin{align*}
\pi_{1}\left(
{\textstyle\sum\limits_{k=1}^{\infty}}
b_{k}^{(n)}S_{k}\left(  a\right)  \right)   &  =b_{1}^{(n)}a_{n_{1,1}}%
+b_{2}^{(n)}a_{n_{2,1}}+b_{3}^{(n)}a_{n_{3,1}}+\cdots=%
{\textstyle\sum\limits_{k=1}^{\infty}}
b_{k}^{\left(  n\right)  }a_{n_{k,1}}\\
\pi_{2}\left(
{\textstyle\sum\limits_{k=1}^{\infty}}
b_{k}^{(n)}S_{k}\left(  a\right)  \right)   &  =b_{1}^{(n)}a_{n_{1,2}}%
+b_{2}^{(n)}a_{n_{2,2}}+b_{3}^{(n)}a_{n_{3,2}}+\cdots=%
{\textstyle\sum\limits_{k=1}^{\infty}}
b_{k}^{\left(  n\right)  }a_{n_{k,2}}\\
&  \frac{{}}{{}}\text{ }\vdots
\end{align*}
by (\ref{777}) we conclude that
\[
S\left(  a\right)  =\left(  \lim_{n\rightarrow\infty}%
{\textstyle\sum\limits_{k=1}^{\infty}}
b_{k}^{\left(  n\right)  }a_{n_{k,1}},\lim_{n\rightarrow\infty}%
{\textstyle\sum\limits_{k=1}^{\infty}}
b_{k}^{\left(  n\right)  }a_{n_{k,2}},\lim_{n\rightarrow\infty}%
{\textstyle\sum\limits_{k=1}^{\infty}}
b_{k}^{\left(  n\right)  }a_{n_{k,3}},\ldots\right)
\]
for all $a=\left(  a_{j}\right)  _{j=1}^{\infty}\in E$. Denote by $xe_{i}$ the
sequence having $x$ in the $i$-th entry and zero elsewhere. Thus, for all
$i\in\mathbb{N}$, there are $k,m\in\mathbb{N}$ such that $i=n_{k,m}%
\in\mathbb{N}_{k}$ and hence%
\[
S\left(  xe_{i}\right)  =(\underset{m-1}{\underbrace{0,\ldots,0}}%
,x\lim_{n\rightarrow\infty}b_{k}^{(n)},0,0,\ldots)\text{,}%
\]
for all $x\in X$. This shows that the all the limits $\lim
\limits_{n\rightarrow\infty}b_{j}^{(n)}$ exist, for all $j\in\mathbb{N}$.
Since $S\neq0$ is continuous and $c_{00}\left(  X\right)  $ is dense in $E$,
we conclude that
\[
\lim_{n\rightarrow\infty}b_{j_{0}}^{(n)}\neq0
\]
for some $j_{0}$. There is no loss of generality in supposing $j_{0}=1$.

Given%
\[
c=\left(  c_{j}\right)  _{j=1}^{\infty}\in E\text{,}%
\]
by (iii), the sequence $d=\left(  d_{j}\right)  _{j=1}^{\infty}$ defined as%
\[
d_{j}=\left\{
\begin{array}
[c]{ll}%
c_{i}\left(  \lim_{n\rightarrow\infty}b_{1}^{(n)}\right)  ^{-1}\text{,} &
\text{if }j=n_{1,i}\text{,}\\
0\text{,} & \text{if }j\notin\mathbb{N}_{1}%
\end{array}
\right.
\]
belongs to $E$ and recalling that
\[
S\left(  a\right)  =\left(  \lim_{n\rightarrow\infty}%
{\textstyle\sum\limits_{k=1}^{\infty}}
b_{k}^{\left(  n\right)  }a_{n_{k,1}},\lim_{n\rightarrow\infty}%
{\textstyle\sum\limits_{k=1}^{\infty}}
b_{k}^{\left(  n\right)  }a_{n_{k,2}},\lim_{n\rightarrow\infty}%
{\textstyle\sum\limits_{k=1}^{\infty}}
b_{k}^{\left(  n\right)  }a_{n_{k,3}},\ldots\right)
\]
for all $a=\left(  a_{j}\right)  _{j=1}^{\infty}\in E$, we have%
\[
S\left(  d\right)  =\left(  \lim_{n\rightarrow\infty}b_{1}^{(n)}d_{1}%
,\lim_{n\rightarrow\infty}b_{1}^{(n)}d_{2},\lim_{n\rightarrow\infty}%
b_{1}^{(n)}d_{3},\lim_{n\rightarrow\infty}b_{1}^{(n)}d_{5},\lim_{n\rightarrow
\infty}b_{1}^{(n)}d_{7},\ldots\right)  =c\text{.}%
\]
This proves Theorem \ref{t1}.

\begin{corollary}
Let $E_{1},\ldots,E_{m}$ be infinite-dimensional Banach spaces and let $E$ be
a $c_{00}$-dense standard Banach sequence space. If there is a surjective
multilinear operator from $E_{1}\times\cdots\times E_{m}$ to $E$, then the set
of all surjective multilinear forms from $E_{1}\times\cdots\times E_{m}$ to
$E$ is $\mathfrak{c}$-lineable.
\end{corollary}

\begin{proof}
Let $\mathcal{S}_{m}$ be the set of all surjective multilinear forms from
$E_{1}\times\cdots\times E_{m}$ to $E$. Let us fix $T_{0}\in\mathcal{S}_{m}$
and consider the set%
\[
W_{m}=\left\{  u\circ T_{0}:u\in V\right\}  \text{,}%
\]
where $W$ is the $\mathfrak{c}$-dimensional subspace of $\mathcal{S}%
\cup\left\{  0\right\}  $ in the proof of Theorem \ref{t1}. It is plain that
$W_{m}$ is a $\mathfrak{c}$-dimensional subspace contained in $\mathcal{S}%
_{m}\cup\left\{  0\right\}  $.
\end{proof}

A similar argument proves that the same holds for polynomials:

\begin{corollary}
Let $E$ be an infinite-dimensional Banach space and $F$ be a $c_{00}$-dense
standard Banach sequence space. If there is a surjective $m$-homogeneous
polynomial from $E$ to $F$, then the set of all surjective $m$-homogeneous
polynomials from $E$ to $F$ is $\mathfrak{c}$-lineable.
\end{corollary}

\section{Spaceability of continuous injective linear operators}

We begin by recalling a more restrictive and somewhat geometric approach to
lineability and spaceability, recently introduced in \cite{cc}. Namely, let
$\alpha$, $\beta$ and $\lambda$ be cardinal numbers and $V$ be a vector space,
with $\dim V=\lambda$ and $\alpha<\beta\leq\lambda$. A set $A\subset V$ is
$\left(  \alpha,\beta\right)  $-lineable if it is $\alpha$-lineable and for
every subspace $W_{\alpha}\subset V$ with $W_{\alpha}\subset A\cup\left\{
0\right\}  $ and $\dim W_{\alpha}=\alpha$, there is a subspace $W_{\beta
}\subset V$ with $\dim W_{\beta}=\beta$ and $W_{\alpha}\subset W_{\beta
}\subset A\cup\left\{  0\right\}  $. Furthermore, if $W_{\beta}$ can be chosen
to be a closed subspace, we say that $A$ is $\left(  \alpha,\beta\right)
$-spaceable. Observe that the ordinary notions of lineability and spaceability
are recovered when $\alpha=0$.

Now we are able to begin the proof of Theorem \ref{t2}. Let us assume that
$\mathcal{I}$ is non empty, and let us fix $T\in\mathcal{I}$. For all
$n\in\mathbb{N}$, let $F^{n}\colon E\rightarrow E$ be the forward $n$-shift
defined in (\ref{hhhhh}). Defining $T_{1}=T$ and $T_{n+1}=F^{n}\circ T$, since
$F^{n}$ is a continuous linear operator, it is immediate that

\begin{itemize}
\item $T_{n}\in\mathcal{I}$, for all $n\in\mathbb{N}$;

\item $\left\Vert T_{n+1}\right\Vert \leq\left\Vert F^{n}\right\Vert
\left\Vert T\right\Vert $, for all $n\in\mathbb{N}$.
\end{itemize}

For the sake of simplicity, we will write $T\left(  x\right)  =\left(  \left(
T\left(  x\right)  \right)  _{n}\right)  _{n=1}^{\infty}$, where $\left(
T\left(  x\right)  \right)  _{n}=\pi_{n}\left(  T\left(  x\right)  \right)  $.
Now, let $\alpha_{1},\ldots,\alpha_{k}$ be non-null scalars, $j_{1}%
<\cdots<j_{k}$ be natural numbers and let us consider the continuous linear
operator%
\[
A=\alpha_{1}T_{j_{1}}+\alpha_{2}T_{j_{2}}+\cdots+\alpha_{k}T_{j_{k}}\text{.}%
\]
Let us show that $A$ is injective. Let $x\in V\backslash\left\{  0\right\}  $.
In this case, $T\left(  x\right)  \neq0$ and, consequently, if $n_{0}$ is the
smallest positive integer such that $\left(  T\left(  x\right)  \right)
_{n_{0}}\neq0$, then $\left(  T_{j_{1}}\left(  x\right)  \right)
_{j_{1}+n_{0}-1}=\left(  T\left(  x\right)  \right)  _{n_{0}}$ is the first
non-null coordinate of $T_{j_{1}}\left(  x\right)  $. Since $j_{1}<j_{i}$,
$i=2,\ldots,k$, we conclude that $\left(  T_{j_{i}}\left(  x\right)  \right)
_{j_{1}+n_{0}-1}=0$ for each $i=2,\ldots,k$ and, therefore, summing coordinate
by coordinate, we can infer that%
\[
\alpha_{1}\left(  T_{j_{1}}\left(  x\right)  \right)  _{j_{1}+n_{0}-1}%
=\alpha_{1}\left(  T\left(  x\right)  \right)  _{n_{0}}%
\]
is the first non-null coordinate of $A\left(  x\right)  $. In particular,
$A\left(  x\right)  \neq0;$ therefore $A$ is injective. Hence, every finite
non trivial linear combination of $T_{n}$ originates an injective linear
operator and, in particular, also originates a non-null linear operator and,
consequently, $\left\{  T_{n}:n\in\mathbb{N}\right\}  $ is linearly
independent. Notice that, at this point, it was established that, whenever $E$
is a Banach sequence space in which the forward $n$-shift $F^{n}$ is
well-defined and continuous, then $\mathcal{I}$ is $\left(  1,\aleph
_{0}\right)  $-lineable.

As we know, $E$ is a standard Banach sequence space over a Banach space $X$
and, by condition (iii) of the definition of standard Banach sequence space,
we have $\left\Vert F^{n}\right\Vert \leq1$ for all $n$. Let $x\in V$ and
$\left(  \lambda_{k}\right)  _{k=1}^{\infty}\in\ell_{1}$. Since%
\begin{align*}
\left\Vert
{\textstyle\sum\limits_{n=1}^{\infty}}
\lambda_{n}T_{n}\left(  x\right)  \right\Vert _{E}  &  \leq%
{\textstyle\sum\limits_{n=1}^{\infty}}
\left\vert \lambda_{n}\right\vert \left\Vert T_{n}\left(  x\right)
\right\Vert _{E}\\
&  \leq%
{\textstyle\sum\limits_{n=1}^{\infty}}
\left\vert \lambda_{n}\right\vert \left\Vert T_{n}\right\Vert \left\Vert
x\right\Vert _{V}\\
&  \leq\left\Vert T\right\Vert \left\Vert x\right\Vert _{V}%
{\textstyle\sum\limits_{n=1}^{\infty}}
\left\vert \lambda_{n}\right\vert ,
\end{align*}
it follows that the linear operator
\begin{align*}
\Phi\colon\ell_{1}  &  \rightarrow\mathcal{L}\left(  V,E\right) \\
\left(  \lambda_{k}\right)  _{k=1}^{\infty}  &  \mapsto%
{\textstyle\sum\limits_{n=1}^{\infty}}
\lambda_{n}T_{n}%
\end{align*}
is well-defined and continuous. Note that the same argument we have used to
prove that the operator $A$ above is injective shows that if $\left(
\lambda_{k}\right)  _{k=1}^{\infty}\in\ell_{1}\backslash\left\{  0\right\}  $,
then $%
{\textstyle\sum\limits_{n=1}^{\infty}}
\lambda_{n}T_{n}$ is injective and, therefore, $\Phi$ is also injective. Thus,%
\[
\operatorname{Im}\left(  \Phi\right)  =\left\{
{\textstyle\sum\limits_{n=1}^{\infty}}
\lambda_{n}T_{n}:\left(  \lambda_{k}\right)  _{k=1}^{\infty}\in\ell
_{1}\right\}  \subset\mathcal{I}\cup\left\{  0\right\}
\]
and
\[
\dim\left(  \operatorname{Im}\left(  \Phi\right)  \right)  =\dim\left(
\ell_{1}\right)  =\mathfrak{c}\text{.}%
\]
Finally, from the arbitrariness of $T\in\mathcal{I}$, we finally conclude the
proof that $\mathcal{I}$ is $\left(  1,\mathfrak{c}\right)  $-spaceable if we
show that%
\[
\overline{\operatorname{Im}\left(  \Phi\right)  }=\overline{\left\{
{\textstyle\sum\limits_{n=1}^{\infty}}
\lambda_{n}T_{n}:\left(  \lambda_{k}\right)  _{k=1}^{\infty}\in\ell
_{1}\right\}  }\subset\mathcal{I}\cup\left\{  0\right\}  \text{.}%
\]
If $S\in\overline{\operatorname{Im}\left(  \Phi\right)  }$, then there are
sequences $\left(  \lambda_{n}^{\left(  k\right)  }\right)  _{n=1}^{\infty}%
\in\ell_{1}$ such that
\[
S=\lim_{k\rightarrow\infty}%
{\textstyle\sum\limits_{n=1}^{\infty}}
\lambda_{n}^{\left(  k\right)  }T_{n}\text{.}%
\]
Since $E$ is a standard Banach sequence space and $S$ is continuous, we have%
\begin{align}
S\left(  x\right)   &  =\lim_{k\rightarrow\infty}%
{\textstyle\sum\limits_{n=1}^{\infty}}
\lambda_{n}^{\left(  k\right)  }T_{n}\left(  x\right) \nonumber\\
&  =\lim_{k\rightarrow\infty}%
{\textstyle\sum\limits_{n=1}^{\infty}}
\lambda_{n}^{\left(  k\right)  }(\overset{n-1\text{ zeros}}{\overbrace
{0,\ldots,0}},\pi_{1}\left(  T\left(  x\right)  \right)  ,\pi_{2}\left(
T\left(  x\right)  \right)  ,\pi_{3}\left(  T\left(  x\right)  \right)
,\ldots)\nonumber\\
&  =\lim_{k\rightarrow\infty}\left[  \lambda_{1}^{\left(  k\right)  }\left(
\pi_{1}\left(  T\left(  x\right)  \right)  ,\pi_{2}\left(  T\left(  x\right)
\right)  ,\pi_{3}\left(  T\left(  x\right)  \right)  ,\pi_{4}\left(  T\left(
x\right)  \right)  \ldots\right)  \right. \nonumber\\
&  \frac{{}}{{}}\text{ \ \ \ \ \ \ \ \ \ \ \ \ \ \ \ \ \ \ }+\lambda
_{2}^{\left(  k\right)  }\left(  0,\pi_{1}\left(  T\left(  x\right)  \right)
,\pi_{2}\left(  T\left(  x\right)  \right)  ,\pi_{3}\left(  T\left(  x\right)
\right)  \ldots\right) \nonumber\\
&  \frac{{}}{{}}\text{ \ \ \ \ \ \ \ \ \ \ \ \ \ \ \ \ \ \ }\left.
+\lambda_{3}^{\left(  k\right)  }\left(  0,0,\pi_{1}\left(  T\left(  x\right)
\right)  ,\pi_{2}\left(  T\left(  x\right)  \right)  ,\ldots\right)
+\cdots\right] \nonumber\\
&  =\lim_{k\rightarrow\infty}\left(  \lambda_{1}^{\left(  k\right)  }\pi
_{1}\left(  T\left(  x\right)  \right)  ,\lambda_{1}^{\left(  k\right)  }%
\pi_{2}\left(  T\left(  x\right)  \right)  +\lambda_{2}^{\left(  k\right)
}\pi_{1}\left(  T\left(  x\right)  \right)  ,\ldots\right) \nonumber\\
&  =\left(  \lim_{k\rightarrow\infty}\lambda_{1}^{\left(  k\right)  }\pi
_{1}\left(  T\left(  x\right)  \right)  ,\lim_{k\rightarrow\infty}\left[
\lambda_{1}^{\left(  k\right)  }\pi_{2}\left(  T\left(  x\right)  \right)
+\lambda_{2}^{\left(  k\right)  }\pi_{1}\left(  T\left(  x\right)  \right)
\right]  ,\ldots\right)  \text{.} \label{www}%
\end{align}

Let us show that $S$ is either identically zero or injective. Assume that $S$
is not injective, i.e. that there exists $w_{0}\in V\backslash\left\{
0\right\}  $ such that $S\left(  w_{0}\right)  =0$. Since $T$ is injective,
then $T\left(  w_{0}\right)  \neq0$; let $n_{0}$ be the smallest positive
integer such that $\pi_{n_{0}}\left(  T\left(  w_{0}\right)  \right)  \neq0$.
Notice that%
\[
0=\pi_{n}\left(  S\left(  w_{0}\right)  \right)  =\lim\limits_{k\rightarrow
\infty}%
{\textstyle\sum\limits_{j=1}^{n}}
\lambda_{n-j+1}^{\left(  k\right)  }\pi_{j}\left(  T\left(  w_{0}\right)
\right)  \text{,}%
\]
for all $n\in\mathbb{N}$. Now, we shall proceed by induction in $m$ to show
that $\lim\limits_{k\rightarrow\infty}\lambda_{m}=0$ for all $m\in\mathbb{N}$.
We can see that%
\begin{align*}
0=\pi_{n_{0}}\left(  S\left(  w_{0}\right)  \right)   &  =\lim
\limits_{k\rightarrow\infty}%
{\textstyle\sum\limits_{j=1}^{n_{0}}}
\lambda_{n_{0}-j+1}^{\left(  k\right)  }\pi_{j}\left(  T\left(  w_{0}\right)
\right)  \\
&  =\lim_{k\rightarrow\infty}\lambda_{1}^{\left(  k\right)  }\pi_{n_{0}%
}\left(  T\left(  w_{0}\right)  \right)  =\left(  \lim_{k\rightarrow\infty
}\lambda_{1}^{\left(  k\right)  }\right)  \pi_{n_{0}}\left(  T\left(
w_{0}\right)  \right)
\end{align*}
and, so,
\[
\lim_{k\rightarrow\infty}\lambda_{1}^{\left(  k\right)  }=0\text{.}%
\]
Assuming by induction hypothesis that
\[
\lim\limits_{k\rightarrow\infty}\lambda_{1}^{\left(  k\right)  }%
=\lim\limits_{k\rightarrow\infty}\lambda_{2}^{\left(  k\right)  }=\cdots
=\lim_{k\rightarrow\infty}\lambda_{m-1}^{\left(  k\right)  }=0\text{,}%
\]
for a certain $m\in\mathbb{N}$, it is obvious that for all $n=1,\ldots,m-1$,
we have%
\[
0=\lim\limits_{k\rightarrow\infty}\lambda_{n}^{\left(  k\right)  }\pi
_{j}\left(  T\left(  w_{0}\right)  \right)
\]
for all $j\in\mathbb{N}$. So, since%
\[
0=\pi_{n_{0}+m-1}\left(  S\left(  w_{0}\right)  \right)  =\lim
\limits_{k\rightarrow\infty}%
{\textstyle\sum\limits_{n=1}^{m}}
\lambda_{n}^{\left(  k\right)  }\pi_{n_{0}+m-n}\left(  T\left(  w_{0}\right)
\right)  \text{,}%
\]
we obtain
\begin{align*}
\lim_{k\rightarrow\infty}\lambda_{m}^{\left(  k\right)  }\pi_{n_{0}}\left(
T\left(  w_{0}\right)  \right)   &  =\lim\limits_{k\rightarrow\infty}\left[
{\textstyle\sum\limits_{n=1}^{m}}
\lambda_{n}^{\left(  k\right)  }\pi_{n_{0}+m-n}\left(  T\left(  w_{0}\right)
\right)  -%
{\textstyle\sum\limits_{n=1}^{m-1}}
\lambda_{n}^{\left(  k\right)  }\pi_{n_{0}+m-n}\left(  T\left(  w_{0}\right)
\right)  \right]  \\
&  =\lim\limits_{k\rightarrow\infty}%
{\textstyle\sum\limits_{n=1}^{m}}
\lambda_{n}^{\left(  k\right)  }\pi_{n_{0}+m-n}\left(  T\left(  w_{0}\right)
\right)  -\lim\limits_{k\rightarrow\infty}%
{\textstyle\sum\limits_{n=1}^{m-1}}
\lambda_{n}^{\left(  k\right)  }\pi_{n_{0}+m-n}\left(  T\left(  w_{0}\right)
\right)  \\
&  =\pi_{n_{0}+m-1}\left(  S\left(  w_{0}\right)  \right)  -%
{\textstyle\sum\limits_{n=1}^{m-1}}
\lim\limits_{k\rightarrow\infty}\lambda_{n}^{\left(  k\right)  }\pi
_{n_{0}+m-n}\left(  T\left(  w_{0}\right)  \right)  \\
&  =0\text{,}%
\end{align*}
and thus%
\[
\lim\limits_{k\rightarrow\infty}\lambda_{m}^{\left(  k\right)  }=0\text{.}%
\]
Hence, $\lim\limits_{k\rightarrow\infty}\lambda_{m}^{\left(  k\right)  }=0$
for all $m\in\mathbb{N}$. Consequently, for all $j\in\mathbb{N}$ and all $x\in
V$, the limit $\lim\limits_{k\rightarrow\infty}\lambda_{m}^{\left(  k\right)
}\pi_{j}\left(  T\left(  x\right)  \right)  $ exists and it is equal to zero.
Thus, by (\ref{www}), we have%
\begin{align*}
S\left(  x\right)   &  =\left(  \lim_{k\rightarrow\infty}\lambda_{1}^{\left(
k\right)  }\pi_{1}\left(  T\left(  x\right)  \right)  ,\lim_{k\rightarrow
\infty}\left[  \lambda_{1}^{\left(  k\right)  }\pi_{2}\left(  T\left(
x\right)  \right)  +\lambda_{2}^{\left(  k\right)  }\pi_{1}\left(  T\left(
x\right)  \right)  \right]  ,\ldots\right)  \\
&  =\left(  \lim_{k\rightarrow\infty}\lambda_{1}^{\left(  k\right)  }\pi
_{1}\left(  T\left(  x\right)  \right)  ,\lim_{k\rightarrow\infty}\lambda
_{1}^{\left(  k\right)  }\pi_{2}\left(  T\left(  x\right)  \right)
+\lim_{k\rightarrow\infty}\lambda_{2}^{\left(  k\right)  }\pi_{1}\left(
T\left(  x\right)  \right)  ,\ldots\right)  \\
&  =0
\end{align*}
for all $x\in V$. Therefore,%
\[
\overline{\operatorname{Im}\left(  \Phi\right)  }\subset\mathcal{I}%
\cup\left\{  0\right\}
\]
and the proof of Theorem \ref{t2} is completed.

The following result is a straightforward consequence of the above proof:

\begin{corollary}
Let $V$ be an infinite dimensional Banach space and let $E$ be a Banach
infinite dimensional sequence space in which the forward shift $F=F^{1}\colon
E\rightarrow E$ is well-defined and continuous. The set%
\[
\mathcal{I}=\left\{  T\colon V\rightarrow E:T\text{\ is linear, continuous and
injective}\right\}
\]
is either empty or $\left(  1,\aleph_{0}\right)  $-lineable.
\end{corollary}

As in the case of surjective polynomials, we have now a similar result for
injective polynomials.

\begin{corollary}
Let $V$ be an infinite dimensional Banach space, let $E$ be a standard Banach
sequence space and consider the set%
\[
\mathcal{I}_{m}=\left\{  P\colon V\rightarrow E:P\text{\ is an injective
}m\text{-homogeneous polynomial}\right\}  \text{.}%
\]
If $\mathcal{I}_{m}\neq\emptyset$, then $\mathcal{I}_{m}$ is $\left(
1,\mathfrak{c}\right)  $-lineable.
\end{corollary}

\begin{proof}
Let $P_{0}\in\mathcal{I}_{m}$ and consider the set%
\[
W_{m}=\left\{  u\circ P_{0}:u\in W\right\}  \text{,}%
\]
where $W$ is a closed $\mathfrak{c}$-dimensional subspace within
$\mathcal{I}\cup\left\{  0\right\}  $ that we can choose containing the
identity operator $\operatorname{id}\colon E\rightarrow E\in W$. Notice that
$P_{0}\in W_{m}$ (we just have to take $u=\operatorname{id}$). Since $W_{m}$
is $\mathfrak{c}$-dimensional and $W_{m}\subset\mathcal{I}_{m}\cup\left\{
0\right\}  $, the proof is done.
\end{proof}


\begin{thebibliography}{99}                                                                                               %


\bibitem {Nacib}N. Albuquerque, L. Bernal-Gonz\'{a}lez, D. Pellegrino, J. B.
Seoane-Sep\'{u}lveda, \emph{Peano curves on topological vector spaces}, Linear
Algebra Appl. \textbf{460 }(2014), 81--96.

\bibitem {Aron}R. M. Aron, V. I. Gurariy, J. B. Seoane-Sep\'{u}lveda,
\emph{Lineability and spaceability of sets of functions on }$\mathbb{R}$,
Proc. Amer. Math. Soc.\textbf{ 133} (2005), 795--803.

\bibitem {bb}R. M. Aron, L. Bernal-Gonz\'{a}lez, P.
Jim\'{e}nez-Rodr\'{\i}guez, G. Mu\~{n}oz-Fern\'{a}ndez, J. B.
Seoane-Sep\'{u}lveda, \emph{On the size of special families of linear
operators}, Linear Algebra Appl. \textbf{544} (2018), 186--205.

\bibitem {book}R. M. Aron, L. Bernal-Gonz\'{a}lez, D. Pellegrino, J. B.
Seoane-Sep\'{u}lveda, \emph{Lineability: The Search for Linearity in
Mathematics}, Monographs and Research Notes in Mathematics. CRC Press, Boca
Raton (2016).

\bibitem {bba}A. Bartoszewicz, M. Bienias, S. G{\l }\k{a}b, T. Natkaniec,
\emph{Algebraic structures in the sets of surjective functions}, J. Math.
Anal. Appl. 441 (2016), no. 2, 574--585.

\bibitem {Bernal}L. Bernal-Gonz\'{a}lez, M. O. Cabrera, \emph{Lineability
criteria, with applications}, J. Funct. Anal. \textbf{266} (2014), 3997--4025.

\bibitem {Pellegrino2}G. Botelho, D. Diniz, V.V. F\'{a}varo, D. Pellegrino,
\emph{Spaceability in Banach and quasi-Banach sequence spaces}, Linear
Algebra. Appl. \textbf{434} (2011), 1255--1260.

\bibitem {BF}G. Botelho, V.V. F\'{a}varo, \emph{Constructing Banach spaces of
vector-valued sequences with special properties}, Michigan Math. J.
\textbf{64} (2015), 539--554.

\bibitem {CFS}D. Cariello, V.V. F\'{a}varo, J. B. Seoane-Sep\'{u}lveda,
\emph{Self-similar functions, fractals and algebraic genericity}, Proc. Amer.
Math. Soc. \textbf{145} (2017), 4151--4159.

\bibitem {cariellojfa}D. Cariello and J. B. Seoane-Sep\'{u}lveda, \emph{Basic
sequences and spaceability in }$\ell_{p}$\emph{ spaces}, J. Funct. Anal.
\textbf{266} (2014), 3797--3814.

\bibitem {FGMR}J. Falc\'{o}, D. Garc\'{\i}a, M. Maestre, P. Rueda,
\emph{Spaceability in norm-attaining sets}, Banach J. Math. Anal. \textbf{11}
(2017), no. 1, 90--107.

\bibitem {cc}V.V. F\'{a}varo, D. Pellegrino, D. Tomaz, \emph{Lineability and
spaceability: a new approach}, Bull Braz Math Soc, New Series \textbf{51}
(2020), 27--46.

\bibitem {SSJB}J. B. Seoane-Sep\'{u}lveda,\emph{ Chaos and lineability of
pathological phenomena in analysis}, ProQuest LLC, Ann Arbor, MI. Thesis
(Ph.D.)-Kent State University (2006).
\end{thebibliography}
\end{document}